\documentclass{aims}
\usepackage{amsmath, amssymb}
 \usepackage{mathrsfs}
  \usepackage{paralist}
  \usepackage{graphics} %% add this and next lines if pictures should be in esp format
  \usepackage{epsfig} %For pictures: screened artwork should be set up with an 85 or 100 line screen
\usepackage{graphicx}  \usepackage{epstopdf}%This is to transfer .eps figure to .pdf figure; please compile your paper using PDFLeTex or PDFTeXify.
 \usepackage[colorlinks=true]{hyperref}
\hypersetup{urlcolor=blue, citecolor=red}

\def\currentvolume{X}
 \def\currentissue{X}
  \def\currentyear{200X}
   \def\currentmonth{XX}
    \def\ppages{X--XX}
     \def\DOI{10.3934/xx.xx.xx.xx}

\newtheorem{theorem}{Theorem}[section]
\newtheorem{corollary}{Corollary}

\newtheorem{lemma}[theorem]{Lemma}
\newtheorem{proposition}{Proposition}

\theoremstyle{definition}

\title[The problem of detecting corrosion] %Use the shortened version of the full title
      {The problem of detecting corrosion by an electric measurement revisited}

\author[Mourad Choulli and Aymen Jbalia]{}


 \email{mourad.choulli@univ-lorraine.fr}
 \email{jbalia.aymen@yahoo.fr}
\begin{document}
\maketitle

% Enter the first author's name and address:
\centerline{\scshape Mourad Choulli}
\medskip
{\footnotesize
 
 \centerline{Institut \'Elie Cartan de Lorraine, UMR CNRS 7502, Universit\'e de Lorraine}
   \centerline{Boulevard des Aiguillettes, BP 70239, 54506 Vandoeuvre les Nancy cedex -}
   \centerline{Ile du Saulcy, 57045 Metz cedex 01, France}
}

\medskip

\centerline{\scshape Aymen Jbalia}
\medskip
{\footnotesize
% please put the address of the first author
 \centerline{Facult\'e des Sciences de Bizerte,}
   \centerline{D\'epartement des Math\'ematiques,}
   \centerline{7021 Jarzouna Bizerte, Tunisie}
} 

% Do not forget to end the {\footnotesize by the sign }

\bigskip

% The name of the associate editor will be entered by an editorial staff
% "Communicated by the associate editor name" is not needed for special issue.
% \centerline{(Communicated by the associate editor name)}

%The abstract of your paper
\begin{abstract}
We establish a logarithmic stability estimate for the problem of detecting corrosion by a single electric measurement. We give a proof based on  an adaptation of the method initiated in \cite{BCJ} for solving the inverse problem of recovering the surface impedance of an obstacle from the scattering amplitude. The key idea consists in estimating accurately  a lower bound of the local $L^2$-norm at the boundary, of the solution of the boundary value problem used in modeling the problem of detection corrosion by an electric measurement.
\end{abstract}

%=========================================================

\section{Introduction}

Let $\Omega$ be a $C^3$-smooth bounded domain of $\mathbb{R}^n$, $n=2,3$, so that the following assumption fulfilled: any two points of $\Omega$ can be joined by a broken line consisting of at most $\ell$ segments, where $\ell$ is a given non negative integer. A domain satisfying this property with $\ell =1$ is nothing but a starshaped domain. 
% and, for sake of simplicity, we assume that $\Omega$ is convex. Of course, the strong convexity condition is not really necessary. We refer to \cite{Chou} for a proof %without the convexity assumption.

\smallskip
We denote the boundary of $\Omega$ by $\Gamma$ and we consider the following boundary value problem (abbreviated to BVP in the sequel)
\begin{eqnarray}\label{1.1}
\left\{
\begin{array}{lll}
\Delta  u =0\quad &\mbox{in}\; \Omega,
\\
\partial _\nu u +q(x)u=g &\mbox{on}\;  \Gamma .
\end{array}
\right.
\end{eqnarray}

In the sequel, $g\in H^{3/2}(\Gamma )$ and $g$ is non identically equal to zero.

\smallskip
For $s\in \mathbb{R}$ and $1\leq r\leq \infty$, we introduce the vector space
\[
B_{s,r}(\mathbb{R}^{n-1} ):=\{ w\in \mathscr{S}'(\mathbb{R}^{n-1});\; (1+|\xi |^2)^{s/2}\widehat{w}\in 
L^r(\mathbb{R}^{n-1} )\},
\]
where $\mathscr{S}'(\mathbb{R}^{n-1})$ is the space of temperated distributions on $\mathbb{R}^{n-1}$ and
$\widehat{w}$ is the Fourier transform of $w$. Equipped with its natural norm 
\[
\|w\|_{B_{s,r}(\mathbb{R}^{n-1} )}:=\|(1+|\xi |^2)^{s/2}\widehat{w}\|_{ L^r(\mathbb{R}^{n-1} )},
\]
$B_{s,r}(\mathbb{R}^{n-1} )$ is a Banach space (it is noted that $B_{s,2}(\mathbb{R}^{n-1} )$ is merely the usual Sobolev
space $H^s(\mathbb{R}^{n-1} )$). By using local charts and a partition of unity, we construct 
$B_{s,r}(\Gamma )$ from $B_{s,r}(\mathbb{R}^{n-1})$ similarly as $H^s(\Gamma )$ is built 
from $H^s(\mathbb{R}^{n-1})$.

\smallskip
To carry out our analysis, we need solutions of the BVP \eqref{1.1} with some smoothness. In order to give sufficient conditions on the coefficient $q$ guaranteeing this smoothness,  we set
\[
\mathscr{Q}=\{q\in B_{3/2,1}(\Gamma );\; q\geq 0\; \mbox{and}\ q\not\equiv 0 \}
\]
and
\[
\mathscr{Q}_M=\{q\in \mathscr{Q};\; \|q\|_{B_{3/2,1}(\Gamma )}\leq M\},
\]
where $M>0$ is a given constant.

\smallskip
By \cite[Theorem 2.3]{Ch1}, for any $q\in \mathscr{Q}$, the BVP \eqref{1.1} has a unique solution $u_q\in H^3(\Omega )$. Moreover,
\begin{equation}\label{5.11}
\|u_q\|_{H^3(\Omega )}\leq C_0\;\; \mbox{for all}\; q\in \mathscr{Q}_M.
\end{equation}
The constant $C_0$ above can depend only on $\Omega$, $g$ and $M$.

\smallskip
\smallskip
It is worthwhile to mention that, from the classical embedding theorems, $H^3(\Omega )$ is continuously embedded in $C^2(\overline{\Omega})$ if $n=2$ and in $C^{1,1/2}(\overline{\Omega})$ if $n=3$. Due to the regularity of $\Omega$, $C^2(\overline{\Omega})$ is continuously embedded in $C^{1,\beta}(\overline{\Omega})$ for any $0\leq \beta \leq 1$. Therefore $H^3(\Omega )$ is continuously embedded in $C^{1,1/2}(\overline{\Omega})$ when $n=2$ or $3$.

\smallskip
A typical example corresponding to the mathematical analysis we develop in the present work is the problem of detecting the corrosion inside a pipe by electric measurements. This is one of the most important topics in engineering, for instance for administering safely a nuclear power station.

\medskip
Usually, in a BVP modeling the problem of detecting corrosion damage by electric measurements the boundary $\Gamma$ consists in two parts: $\Gamma =\overline{\Gamma _a\cup \Gamma _i}$, $\Gamma _a$ and $\Gamma _i$ being two disjoint relatively open subsets of $\Gamma$. Here, $\Gamma _a$ corresponds to the part of the boundary accessible to measurements and $\Gamma _i$ is the inaccessible part of the boundary where the corrosion damage may occur.  

\smallskip
Henceforth, we assume that the current flux $g$ satisfies $\mbox{supp}(g)\subset \Gamma _a$. The function $q$ in \eqref{1.1} is known as the corrosion coefficient and it is naturally supported on $\Gamma _i$. This motivate the introduction of the following subset of $\mathscr{Q}_M$:
\[
\mathscr{Q}_M^0=\{q\in \mathscr{Q}_M;\; \mbox{supp}(q)\subset \Gamma _i\}.
\]

We are interested in the stability issue for the  problem consisting in the determination of the boundary coefficient $q$ from the boundary measurement $u_q{_{|\gamma}}$, where $\gamma$ is an open subset of the accessible sub-boundary $\Gamma _a$. In the sequel, we assume that $\gamma$ does not meet $\mbox{supp}(g)$:
\[
\gamma \subset \Gamma _a\setminus \mbox{supp}(g).
\]

\smallskip
We aim  to prove the following theorem.  We only sketch the main steps of the proof since the most intermediate results consist in an adaptation of the ones already proved in \cite{BCJ} (see also \cite{Chou}).

\begin{theorem}\label{theorem1.1}
There exist four constants $C_i>0$ and $\sigma _i>0$, $i=0,1$ so that, for any $q,\widetilde{q} \in \mathscr{Q}_M^0\cap C^{0,\alpha} (\Gamma)$, 
\[
\|q-\widetilde{q}\|_{L^\infty (\Gamma )}\le C_0 \Psi \left(C_1 \Phi \left(\| u-\widetilde{u}\|_{L^2(\gamma )}\right)\right),
\]
where $u=u_q$, $\widetilde{u}=u_{\widetilde{q}}$ and
\begin{align*}
\Psi  (\rho )=  \left|\ln \ln \rho \right|^{-\sigma _1}+\rho ,\;\; \rho >0,
\\
\Phi (\rho )= \left|\ln  \rho \right|^{-\sigma _2}+\rho ,\;\; \rho >0.
\end{align*}
\end{theorem}

%\begin{theorem}\label{theorem1.1}
%There are three positive constants $A$,  $B$ and $\sigma$ satisfying: for any $q\in \mathscr{Q}_M^0\cap C^\alpha (\Gamma)$, we find $\epsilon=\epsilon (q)$ so that %for all $\widetilde{q}\in \mathscr{Q}_M^0\cap C^\alpha (\Gamma)$ such that $\|q-\widetilde{q}\|_{L^\infty (\Gamma _i )}\leq \epsilon$, 
%\[
%\|q-\widetilde{q}\|_{L^\infty (\Gamma )}\leq \frac{A}{\left|\ln \left|\ln \left(B\|u-\widetilde{u}\|_{L^2(\gamma )}\right)\right|\right|^\sigma},
%\]
%with $u=u_q$ and $\widetilde{u}=u_{\widetilde{q}}$.
%\end{theorem}

Theorem \ref{theorem1.1} can be seen as a completion of the results established in \cite{CCL} in dimension two and in \cite{BCC} in dimensions two and three. We note that in the above mentioned works the difference of $q-\widetilde{q}$ is only estimated in a compact subset of $\{x\in \Gamma _i;\; u_q(x)\not=0\}$. However, there is a counterpart in estimating $q-\widetilde{q}$ in the whole $\Gamma_i$. The stability estimates in  \cite{CCL} and \cite{BCC} are of single logarithmic type, while the estimate in Theorem \ref{theorem1.1} is of triple logarithmic type. This means that the stability deteriorates near the points where the solution of the BVP \eqref{1.1} vanishes. 

\smallskip
There is a wide literature treating the problem of detecting corrosion by electric measurements. We refer to \cite{ADR, CFJL, CJ, CCY, Ch2, Ch3, FI, In, Si2, Si3} where various type of stability estimates are given. We just quote these few references, but of course there are many others. A neighbor problem is the one consisting in the determination of the surface impedance of an obstacle from the scattering amplitude (e.g. \cite{ASV, BCJ, Si1} and the reference therein).

\smallskip
To our knowledge, the existing results on stability including the vicinity of the zeroes of the solution, of the BVP under consideration, do not give self contained proofs. They always refer to several previous works. Therefore it is very hard to recover completely the proofs. Although our result seems to be weaker than some of these existing results, our method is direct and it is based only on an elementary Carleman inequality. In fact  our result is not really comparable  with those of the literature because in the existing results $q$ is only estimated on a fixed subset of $\Gamma _i$.

\smallskip
The rest of this text consists in two sections. In section 2 we estimate accurately a lower bound of the local $L^2$-norm at the boundary, of the solution of the BVP \eqref{1.1}. We show, step by step, how we adapt the method in \cite{BCJ} to the present problem. Section 3 is devoted the proof of Theorem \ref{theorem1.1}.

\smallskip
Unless otherwise specified, all the constants we use in the sequel depend only on data.

%=========================================

\section{Lower bound for $L^2$-norm at the boundary}

 We first note that as  $\Omega$ is $C^3$-smooth, it has the uniform interior cone property, abbreviated to the {\bf UICP} in the sequel. That is there are $R>0$ and $\theta \in ]0,\pi[$ satisfying, for all $\widetilde{x}\in \Gamma$, we find $\xi \in \mathbb{R}^n$ such that $|\xi |=1$ and
\[
\mathcal{C}(\widetilde{x})=\{x\in \mathbb{R}^n;\; |x-\widetilde{x}|<R\; \mbox{and}\; (x-\widetilde{x})\cdot \xi >|x-\widetilde{x}|\cos \theta \}\subset \Omega .
\]

The domain $\Omega$  satisfies also the uniform exterior sphere property, abbreviated to {\bf UESP} in the sequel: there exists $\rho >0$ so that, to any $\widetilde{x}\in \Gamma $ corresponds $x_0=x_0(\widetilde{x})\in \mathbb{R}^n\setminus \overline{\Omega}$ for which 
\[
B(x_0,\rho )\cap \Omega =\emptyset\;\;  \textrm{and}\;\; B(x_0,\rho )\cap \overline{\Omega}=\{\widetilde{x}\}.
\]
Additionaly, $\xi$ is defintion of {\bf UICP} can be chosen as follows $\xi =(\widetilde{x}-x_0)/|\widetilde{x}-x_0|$.

%We prove the following theorem.

%\begin{theorem}\label{theorem2.1}
%Let $M>0$, there is $c>0$ so that: for all $q\in \mathscr{Q}^0_M$ and $\widetilde{x}\in \Gamma$, 
%\[
%e^{-ce^{\frac{c}{r}}}\leq \|u_q\|_{L^2(B(\widetilde{x},r)\cap \Gamma )},\;\; 0<r\leq r^\ast ,
%\]
%where $r^\ast$ is a constant that can depend on $q$.
%\end{theorem}

%We need some preliminary results before proving Theorem \ref{theorem2.1}. 

%\smallskip
%For sake of simplicity, we assume in the sequel that $\Omega$ is in addition starshaped. From the proof of Proposition \ref{proposition3.1} below, one can see that %the extension to the case where $\Omega$ is multiply-starshaped  is obvious.

\smallskip
 For $\delta >0$, we set
\[
\Omega ^\delta =\{x\in \Omega ;\; \mbox{dist}(x,\Gamma )>\delta \}
\]
and we recall the following useful three sphere inequality for  the $H^1$-norm.
\begin{lemma}\label{lemma3.1}
There exist $C>0$ and $0<s <1$ so that: for all  $u\in H^1(\Omega )$ satisfying $\Delta u=0$ in $\Omega$, $y\in \Omega$ and $0<r< \frac{1}{3 }\mbox{dist}(y,\Gamma )$,  
\[
r\|u\|_{H^1(B(y,2r))}\leq C\|u\|_{H^1(B(y,r))}^s\|u\|_{H^1(B(y,3r))}^{1-s}.
\]
\end{lemma}
We refer to \cite{BCJ} for a proof. The case of a general divergence form operator is detailed in \cite{CT} and \cite{Chou}.

\begin{proposition}\label{proposition3.1}
Let $s$ be as in the previous lemma and fix $M>0$. There exists a constant $C>0$ so that for any $\delta >0$,  $u\in H^1(\Omega )$ satisfying $\Delta u=0$, $\|u\|_{H^1(\Omega )}\leq M$ and,  $x$, $y\in \Omega ^\delta$, 
 \[
%e^{-ce^{\frac{c}{r}}}\|u\|_{H^1(B(x ,r))}\leq \|u\|_{H^1(B(y ,2r))},\;\; 0<r<\delta /6 .
(Cr)^{1-s^N}\|u\|_{H^1(B(x,r))}\leq \|u\|_{H^1(B(y ,2r))}^{s^N},\;\; 0<r<\delta /6 ,
 \]
 where $N$ is the smallest integer satisfying $|x-y|-2Nr\le r$.
\end{proposition}

\begin{proof}
The proof is given for $\ell=1$. For an arbitrary $\ell$, the proof is quite similar with slight adaptation. Set
 \[
 d=|y-x |,\;\;  \eta =\frac{y-x}{|y-x|}
 \]
 and consider the sequence, where $0<2r<d$,
 \[
 x_k=y-k(2r)\eta ,\;\; k\geq 1.
 \]
Clearly,
 \[
 |x_k-x|=d-k(2r).
 \]
Let $N$ be the smallest integer such that $d-N(2r)\leq r$, or equivalently
 \[
\frac{d}{2r}-\frac{1}{2}\leq N <\frac{d}{2r}+\frac{1}{2}.
 \]
 
 By Lemma \ref{lemma3.1}, there exist $C>0$ and $0<s<1$ so that 
 \[
Cr\|u\|_{H^1(B(x_1,r))}\leq \|u\|_{H^1(B(y,r))}^{s},
\]
where we take into account that $B(x_1,r)\subset B(y,2r)$.

\smallskip
Similarly, we have 
 \[
Cr\|u\|_{H^1(B(x_{k+1},r))}\leq \|u\|_{H^1(B(x_k,r))}^{s},\;\; k\ge 1.
\]

Hence, an induction argument yields
 \begin{equation}\label{3.1}
(Cr)^{1-s^N}\|u\|_{H^1(B(x_N,2r))}\leq \|u\|_{H^1(B(y ,2r))}^{s^N}.
 \end{equation}

 Since $|x_N-x|=d-N(2r)\leq r$, $B(x,r)\subset B(x_N, 2r)$. Whence \eqref{3.1} entails
  \begin{equation}\label{3.4}%\label{3.2}
(Cr)^{1-s^N}\|u\|_{H^1(B(x,r))}\leq \|u\|_{H^1(B(y ,2r))}^{s^N}.
 \end{equation}
% Or equivalently
 %\begin{equation}\label{3.4}
 %(Cr)^\kappa\|u\|_{H^1(B(x ,r))}\leq \|u\|_{H^1(B(y ,2r))},
 %\end{equation}
 %with
 %\[
% \kappa =te^{N|\ln s|}.
 %\]
 %The expected inequality follows 

%Assume that $r$ is sufficiently small in such a way that $Cr<1$. Then the expected inequality follows from \eqref{3.4} since
%\[
%(Cr)^{\frac{1-s^N}{s^N}}=e^{\left(-e^{N|\ln s|}+1\right)|\ln (Cr)|}\ge e^{\left(-e^{\left(\frac{D}{2r}+\frac{1}{2}\right)|\ln s|}+1\right)|\ln (Cr)|},
%\]
%with $D=\mbox{diam}(\Omega )$.
 \end{proof}

 We recall that according to Caccioppoli's inequality, for all $u\in H^1(\Omega )$ satisfying $\Delta u=0$ in $\Omega$ and $x\in \Omega$,   
 \[
 \|\nabla u\|_{L^2(B(x,2r))^n}\leq Cr^{-1}\|u\|_{L^2(B(x,3r))},
 \]
 for $0<3r<\textrm{dist}(x,\Gamma )$.
 
 \smallskip
 Therefore the following corollary is immediate from Proposition \ref{proposition3.1}.
 
 \begin{corollary}\label{corollary3.1}
Let $s$ be as in the previous lemma and fix $M>0$. There exists a constant $C>0$ so that for any $\delta >0$,  $u\in H^1(\Omega )$ satisfying $\Delta u=0$, $\|u\|_{H^1(\Omega )}\leq M$ and,  $x$, $y\in \Omega ^\delta$, 
 \[
%e^{-ce^{\frac{c}{r}}}\|u\|_{H^1(B(x ,r))}\leq \|u\|_{H^1(B(y ,2r))},\;\; 0<r<\delta /6 .
(Cr)^{1-s^N}\|u\|_{H^1(B(x,r))}\leq \|u\|_{L^2(B(y ,3r))}^{s^N},\;\; 0<r<\delta /6 ,
 \]
 where $N$ is the smallest integer satisfying $|x-y|-2Nr\le r$.
 \end{corollary}

%===================================================================== 
 %By an elementary continuity argument, we get from this corollary
% \begin{corollary}\label{corollary3.2}
% Fix $0<\eta >0\le M$. There is $c>0$ with the property that, to any $u\in H^1(\Omega )$, satisfying
 
 %\smallskip
% \noindent
% $\Delta u=0$, $\|u\|_{H^1(\Omega )}\leq M$ and there is $\widehat{x}\in \Gamma$ such that $u\in C(B(\widehat{x}, \widehat{r})\cap \Omega )$, for some $\widehat{r}%>0$, and $|u(\widehat{x} )|\geq \eta$, 
 
 %\smallskip
% \noindent
 %corresponds $\delta >0$ and $r_\delta >0$ for which, for all $y\in \Omega ^\delta$,
%  \[
%e^{-ce^{\frac{c}{r}}}\leq \|u\|_{L^2(B(y ,r))},\;\; 0<r<r_\delta .
% \]
% Note here that $\delta$ and $r_\delta$ may depend also on $u$.
 %\end{corollary}
 %==========================================================================
% As  $\Omega$ is $C^n$-smooth, it has also the uniform interior cone property, abbreviated to the {\bf UICP} in the sequel. That is there are $R>0$ and $\theta %\in ]0,2\pi[$ satisfying, for all $\widetilde{x}\in \Gamma$, we find $\xi \in \mathbb{R}^n$ such that $|\xi |=1$ and
%\[
%\mathcal{C}(\widetilde{x})=\{x\in \mathbb{R}^n;\; |x-\widetilde{x}|<R\; \mbox{and}\; (x-\widetilde{x})\cdot \xi >|x-\widetilde{x}|\cos \theta \}\subset \Omega .
%\]
 
 \begin{corollary}\label{corollary3.2}
Fix $0<\alpha <1$ and $0<\eta \le M$, and set $\delta _0=\min ((\eta /(2M))^{1/\alpha},R/2)$. There exists a constant $c>0$ so that, for any $u\in H^1(\Omega )\cap C^{0,\alpha} (\overline{\Omega})$ satisfying
 $\Delta u=0$, $\|u\|_{H^1(\Omega )\cap C^{0,\alpha} (\overline{\Omega})}\leq M$ and $|u(\widetilde{x})|\ge \eta$, for some $\widetilde{x}\in \Gamma$, and $0<\delta \le \delta _0$,
 \[
e^{-ce^{\frac{c}{r}}}\leq \|u\|_{L^2(B(y ,r))},\;\; y\in \Omega^ {(\delta /2)\sin \theta} ,\; 0<r< (\delta /3)\sin \theta.
 \]
Here $R$ and $\theta$ are as in the definition of the {\bf UICP}.
 \end{corollary}
 
\begin{proof}
Let $u\in H^1(\Omega )\cap C^{0,\alpha} (\overline{\Omega})$, satisfying $\Delta u=0$, $\|u\|_{H^1(\Omega )\cap C^{0,\alpha} (\overline{\Omega})}\leq M$ and $|u(\widetilde{x})|\ge \eta$, for some $\widetilde{x}\in \Gamma$. If $z\in B(\widetilde{x},\delta _0)$ then
\begin{align*}
|u(z)|&\ge |u(\widetilde{x})|-|u(\widetilde{x})-u(z)|
\\
&\ge |u(\widetilde{x})|-|\widetilde{x}-z|^\alpha [u]_\alpha
\\
&\ge \eta -\delta ^\alpha M=\eta /2.
\end{align*}

Let $\xi =\xi (\widetilde{x})$  be as in the definition of {\bf UICP} and  set $x=\widetilde{x}+(\delta _0/2)\xi$. Since $B(x,(\delta _0/2)\sin \theta )\subset \mathcal{C}(x)\subset \Omega$, $x\in \Omega^{(\delta _0/2)\sin \theta}$. If $\delta \le \delta _0$, we get by applying Corollary \ref{corollary3.1}, for all $y\in \Omega^{(\delta /2)\sin \theta}$,
\begin{equation}\label{5.101}
(Cr)^{1-s^N}\|u\|_{H^1(B(x ,r/3))}\leq \|u\|_{L^2(B(y ,r/3))}^{s^N},\;\; 0<r<(\delta /3)\sin \theta ,
\end{equation}
where $N$ is the smallest integer so that $|x-y|-2Nr\leq r$.

\smallskip
But
\begin{equation}\label{5.102}
\|u\|_{H^1(B(x ,r/3))}\ge \|u\|_{L^2(B(x,r/3))}\ge (\eta /2)|\mathbb{S}^{n-1}|(r/3)^{n/2},
\end{equation}
due to the fact that $B(x,r/3)\subset B(\widetilde{x},\delta _0)$. Hence, after some computations, \eqref{5.101} and \eqref{5.102} yield
\begin{equation}\label{8.1}
(Cr)^{\varrho (N)}\le \|u\|_{L^2(B(y ,r))},\;\; 0<r<(\delta /3)\sin \theta ,
\end{equation}
with
\[
\varrho (N)= \frac{n/2+1}{s^N}-1.
\]
Shortening $C$ if necessary, we can assume that $Cr<1$. Hence
\begin{equation}\label{8.2}
(Cr)^{\varrho (N)}=e^{\left(-(n+1/2)e^{N|\ln s|}+1\right)|\ln (Cr)|}\ge e^{\left(-e^{\left((n+1/2)\left(\frac{D}{2r}+\frac{1}{2}\right)\right)|\ln s|}+1\right)|\ln (Cr)|}.
\end{equation}
Here $D=\mbox{diam}(\Omega )$.

\smallskip
The expected inequality is derived in a straightforward manner from \eqref{8.1} and \eqref{8.2}.
\end{proof}

If $\widetilde{x}\in \Gamma$, and $x_0=x_0(\widetilde{x})$ and $\rho$ are as in the definition of {\bf UESP}, we set
\[
\mathcal{B}(\widetilde{x},r)=B(x_0,r+\rho ),\;\; r>0.
\]

 \begin{theorem}\label{theorem5.1}
 Fix $0<\alpha <1$ and $0<\eta \le M$. Let $\delta _0$ be as in Corollary \ref{corollary3.2}. There exists a constant $c>0$ so that, for any $u\in H^2(\Omega )\cap C^{0,\alpha }(\overline{\Omega})$ satisfying
 $\Delta u=0$, $\|u\|_{H^2(\Omega )\cap C^{0,\alpha} (\overline{\Omega})}\leq M$ and $|u(\widetilde{x})|\ge \eta$, for some $\widetilde{x}\in \Gamma$, 
 \[
e^{-ce^{\frac{c}{r}}}\le \|u\|_{L^2(\mathcal{B}(\widehat{x},r)\cap \Gamma )}+\|\nabla v\|_{L^2(\mathcal{B}(\widehat{x},r)\cap \Gamma )^n} ,\; 0<r< r_0=\min (R,8\delta _0/\sin \theta ),
 \]
 for any $\widehat{x}\in \Gamma$, where $R$ and $\theta$ are as in the definition of {\bf UICP}
 \end{theorem}
 
\begin{proof}
From \cite[Corollary 3.1]{BCJ}, there exist two constants $C_0>0$ and $0<\beta <1/2$ so that, for any $v\in H^2(\Omega )$ satisfying $\Delta v=0$, $\widehat{x}\in \Gamma$ and $r>0$,
\begin{align}
&C_0r^2\|v\|_{H^1(\mathcal{B}(\widehat{x},\frac{r}{4})\cap\Omega )}\label{5.103}
\\
&\hskip 2cm \le \|v\|_{H^2(\Omega )}^{1-\beta}\left(\|u\|_{L^2(\mathcal{B}(\widehat{x},r)\cap \Gamma )}+\|\nabla v\|_{L^2(\mathcal{B}(\widehat{x},r)\cap \Gamma )^n}\right)^{\beta }.\nonumber
\end{align}

Let $u\in H^2(\Omega )\cap C^{0,\alpha} (\overline{\Omega})$, satisfying $\Delta u=0$, $\|u\|_{H^2(\Omega )\cap C^{0,\alpha} (\overline{\Omega})}\leq M$ and $|u(\widetilde{x})|\ge \eta$, for some $\widetilde{x}\in \Gamma$.

\smallskip
For $0<r<R$, let $y=\widehat{x}+(r/4)\xi$ with $\xi =\xi (\widehat{x})$. Then it is straightforward to check that $B(y,(r/4)\sin \theta )\subset \Omega ^{(r/4)\sin \theta}\cap \mathcal{B}(\widehat{x},r)$. Therefore by \eqref{5.103}
\begin{equation}\label{5.104}
Cr^2\|u\|_{L^2(B(y,(r/4)\sin \theta ))}\le \left(\|u\|_{L^2(\mathcal{B}(\widehat{x},r)\cap \Gamma )}+\|\nabla v\|_{L^2(\mathcal{B}(\widehat{x},r)\cap \Gamma )^n}\right)^{\beta }.
\end{equation}

On the other hand from Corollary \ref{corollary3.2}
\begin{equation}\label{5.105}
e^{-ce^{\frac{c}{\rho }}}\leq \|u\|_{L^2(B(y ,\rho ))},\;\;  0<\rho < (r /6)\sin \theta,
\end{equation}
provided that $\rho \le \delta_0$, where $\delta_0$ is defined in Corollary \ref{corollary3.2}.

\smallskip
Combine \eqref{5.104} and \eqref{5.105} with $\rho= (r/8)\sin \theta$ in order to get
\[
e^{-ce^{\frac{c}{r}}}\le \|u\|_{L^2(\mathcal{B}(\widehat{x},r)\cap \Gamma )}+\|\nabla v\|_{L^2(\mathcal{B}(\widehat{x},r)\cap \Gamma )^n}.
\]
\end{proof}

One can proceed as in the proof of  \cite[Theorem 4.1]{BCJ} to derive from the previous theorem the following corollary.

\begin{corollary}\label{corollary5.101}
Let $\Gamma _0$ be an open non empty subset of $\Gamma$. Fix $0<\alpha  <1$, $\Lambda _0>0$ and $0<\eta \le M$, and let $r_0$ be as in Theorem \ref{theorem5.1}. There exists a constant $c>0$ so that, for any $u\in H^2(\Omega )\cap C^{0,\alpha} (\overline{\Omega})$ satisfying
\begin{eqnarray*}%\label{3.37}
\left\{
\begin{array}{lll}
\Delta u=0\;\; \mbox{in}\; \Omega ,
\\
|\partial _\nu u|\leq \Lambda _0|u|\;\; \mbox{on}\;\; \Gamma _0 ,
\\
|u(\widetilde{x})|\ge \eta \;\; \mbox{for some}\; \widetilde{x}\in \Gamma ,
\\
\|u\|_{H^2(\Omega )\cap C^{0,\alpha} (\overline{\Omega})}\leq M,
\end{array}
\right.
\end{eqnarray*}
and $\widehat{x}\in \Gamma_0$,
\[
e^{-ce^{\frac{c}{r}}}\le \|u\|_{L^2(\mathcal{B}(\widehat{x},r)\cap \Gamma )},\;\; 0<r<r _0/2.
\]
\end{corollary}

Therefore, we can mimic the proof of \cite[Proposition 4.1]{BCJ}  to get the following result.
\begin{proposition}\label{proposition5.101}
Let $\Gamma _0$ be an open non empty subset of $\Gamma$. Fix $0<\alpha ,\beta <1$, $\Lambda _0>0$ and $0<\eta \le M$. There exist constants $C>0$ and $\sigma >0$ so that, for any $u\in H^2(\Omega )\cap C^{0,\alpha} (\overline{\Omega})$ satisfying
\begin{eqnarray*}%\label{3.37}
\left\{
\begin{array}{lll}
\Delta u=0\;\; \mbox{in}\; \Omega ,
\\
|\partial _\nu u|\leq \Lambda _0|u|\;\; \mbox{on}\;\; \Gamma _0 ,
\\
|u(\widetilde{x})|\ge \eta \;\; \mbox{for some}\; \widetilde{x}\in \Gamma ,
\\
\|u\|_{H^2(\Omega )\cap C^{0,\alpha} (\overline{\Omega})}\leq M,
\end{array}
\right.
\end{eqnarray*}
and for any $f\in C^{0,\beta}(\overline{\Gamma }_0)$ with $\|f\|_{C^{0,\beta}(\overline{\Gamma }_0)}\le M$,
\[
\|f\|_{L^\infty (\Gamma _0)}\leq C\left(\left|\ln \ln \left(\|fu\|_{L^\infty (\Gamma _0)}\right)\right|^{-\sigma}+\|fu\|_{L^\infty (\Gamma _0)}\right).
\]
\end{proposition}

%===============================================================================

\section{Proof of the stability estimate}

\begin{lemma}\label{lemma5.101}
There exist  a constant $\eta >0$, that can depend on $M$ and $g$, with the property that, for any $q\in \mathscr{Q}_M^0$, one finds $\widetilde{x}\in \overline{\gamma}$ so that $|u_q(\widetilde{x})|\ge \eta$.
\end{lemma}
    
\begin{proof}
Let $\gamma _0\Subset \gamma$ and $C_0$ be the constant in \eqref{5.11}. By \cite[Corollary 1]{Bo}, there exist three constants $C>0$, $\beta >0$ and $\delta >0$ so that for any $\lambda >0$ 
\[
C\|\lambda g\|_{L^\infty (\Gamma )}\le C_0\left| \ln \left(C_0^{-1}(\|u_q(\lambda g)\|_{L^2(\gamma _0)}+\|\nabla u_q(\lambda g)\|_{L^2(\gamma _0)}\right)\right|^{-\beta},
\]
whenever $\|u_q(\lambda g)\|_{L^2(\gamma _0)}+\|\nabla u_q(\lambda g)\|_{L^2(\gamma _0)}\le \delta$.

\smallskip
Using a cutoff function and an interpolation inequality we obtain
\[
\|u_q(\lambda g)\|_{L^2(\gamma _0)}+\|\nabla u_q(\lambda g)\|_{L^2(\gamma _0)}=\|u_q(\lambda g)\|_{H^1(\gamma _0)}\leq C_1\|u_q(\lambda g)\|_{L^2(\gamma )}.
\]
Then the choice of $\lambda >0$ so that $C_1C_0\lambda \leq \delta$ gives 
\[
C\lambda \|g\|_{L^\infty (\Gamma )}\leq C_0\left| \ln \left(C_0^{-1}\lambda \|u_q(g)\|_{L^2(\gamma )}\right)\right|^{-\beta}
\]
Hence
\[
C_0e^{-C_1\|g\|_{L^\infty (\Gamma )}^{-1/\beta}}\le \|u_q(g)\|_{L^2(\gamma )}\le |\gamma |^{1/2}\| u_q\|_{L^\infty (\gamma )}.
\]
This estimate implies that there exists $\widetilde{x}\in \overline{\gamma}$ so that
\[
\eta = |\gamma |^{-1/2}C_0e^{-C_1\|g\|_{L^\infty (\Gamma )}^{-1/\beta}}\le |u_q(\widetilde{x})|.
\]

\end{proof}

This lemma in combination with Proposition \ref{proposition5.101} yields

\begin{proposition}\label{proposition5.102}
Let $0<\beta <1$. There exist two constants $C>0$  and $\sigma >0$ so that, for any $q\in \mathscr{Q}_M^0$ and $f\in C^{0,\beta}(\Gamma )$ with $\|f\|_{C^{0,\beta}(\Gamma )} \le M$,
\begin{equation}\label{5.12}
\|f\|_{L^\infty (\gamma )}\leq C\left(\left|\ln \ln \left(\|fu_q\|_{L^\infty (\gamma )}\right)\right|^{-\sigma}+\|fu_q\|_{L^\infty (\gamma )}\right).
\end{equation}
\end{proposition}

\begin{proof}[Proof of Theorem \ref{theorem1.1}]
Let $v=\widetilde{u}-u$. Since $\Delta v=0$. The stability estimate for the Cauchy problem in \cite{Bo} (see also \cite{Chou}) yields
\begin{equation}\label{5.17}
\| v\|_{W^{1,\infty}(\Gamma )}\leq C\left(\left|\ln (\|v\|_{L^2(\gamma )})\right|^{-\beta}+\|v\|_{L^2(\gamma )}\right),
\end{equation}
for some constants $C>0$ and $\beta >0$.

But
\begin{equation}\label{5.18}
(q-\widetilde{q})u=\partial_\nu v+\widetilde{q}v.
\end{equation}

Hence \eqref{5.17}  yields
\begin{equation}\label{5.19}
\|(q-\widetilde{q})u\|_{L^\infty (\Gamma )}\le C\left(\left|\ln (\|v\|_{L^2(\gamma )})\right|^{-\beta}+\|v\|_{L^2(\gamma )}\right).
\end{equation}
In light of \eqref{5.17}, we end up getting
\[
\|q-\widetilde{q}\|_{L^\infty (\Gamma )}\le C_0 \Psi \left(C_1 \Phi \left(\| u-\widetilde{u}\|_{L^2(\gamma )}\right)\right).
\]
The proof of Theorem \ref{theorem1.1} is then complete.
\end{proof}

%========================================

\vskip .5cm

\end{document}